\documentclass[preprint,12pt]{elsarticle}

\usepackage{fullpage}   
\usepackage{amsmath}
\usepackage{amsfonts}
\usepackage{amssymb}
\usepackage{amsthm}
\usepackage{subcaption}
\usepackage{graphicx}
\usepackage{color}
\usepackage{pinlabel}
\usepackage{extarrows}

\usepackage{pinlabel}
\usepackage{hyperref}

\newtheorem{theorem}{Theorem}
\newtheorem*{whitney}{Whitney's 2-Isomorphism Theorem}

\newtheorem{lemma}{Lemma}
\newtheorem{corollary}{Corollary}
\theoremstyle{definition}
\newtheorem{definition}{Definition}

\theoremstyle{remark}

\newcommand{\ba}{\setminus}

\newcommand{\B}{\mathcal{B}}
\newcommand{\F}{\mathcal{F}}
\newcommand{\mat}[1]{\mathbf{#1}}

\newcommand{\GF}{\mathrm{GF}}
\newcommand{\G}{\mathbb{G}}

\newcommand{\bH}{\mathbb{H}}

\newcommand{\calA}{\mathcal{A}}

\newcommand{\iso}{\cong}

\journal{}

\begin{document}

\begin{frontmatter}

\title{A 2-isomorphism theorem for delta-matroids}

\author[1]{Iain Moffatt}
\address[1]{Department of Mathematics,
Royal Holloway,
University of London,
Egham,
TW20~0EX,
United Kingdom}
\ead{iain.moffatt@rhul.ac.uk}

\author[2]{Jaeseong Oh}
\address[2]{Department of Mathematical Sciences,
Seoul National University,
Seoul, 
GwanAkRo 1, Gwanak-Gu,
Korea}
\ead{jaeseong\_oh@snu.ac.kr}
\begin{abstract}
 Whitney's 2-Isomorphism Theorem characterises when two graphs have  isomorphic cycle matroids. We present an analogue of this theorem for graphs embedded in surfaces by characterising when two graphs  in surface have isomorphic delta-matroids.
\end{abstract}


\end{frontmatter}

\section{Introduction}\label{s1}

There is a well-known  symbiotic relationship between graph theory and matroid theory, with each area informing  the other. (A good introduction to this relationship can be found in~\cite{MR1850709}). The transition between graphs and matroids is usually  through  cycle matroids. However,  care must be taken when moving between graphs and matroids since there is not a 1-1 correspondence between the sets of graphs and cycle matroids --- different graphs can give rise to the same cycle matroids. 
\emph{Whitney's 2-Isomorphism  Theorem}~\cite{whitney} (see also~\cite{MR558452,MR816055}) characterises this indeterminacy by determining when two graphs have isomorphic cycle matroids:
\begin{whitney}\label{wt}
Let $G$ and $H$ be graphs, with cycle matroids  $C(G)$ and $C(H)$. 
 Then $C(G)$ and  $C(H)$ are isomorphic matroids if and only if $G$ and $H$ are related by isomorphism, vertex identification, vertex cleaving, or Whitney twisting.
\end{whitney}
The graph moves in the theorem are described in Figure~\ref{ddwe} (their formal definitions, which are not needed here, can be found in Chapter~5 of~\cite{Oxley}). 
Graphs $G$ and $H$ that are related as in the theorem statement are said to be  \emph{2-isomorphic}.

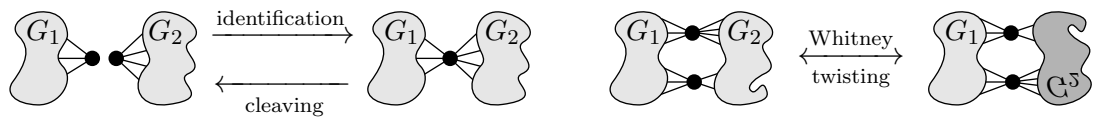
\begin{figure}[!ht]
    \centering
     \labellist
\small\hair 2pt
\pinlabel {$G_1$} at 14 29
\pinlabel {$G_2$} at 60 29
\endlabellist  
\includegraphics[scale=1]{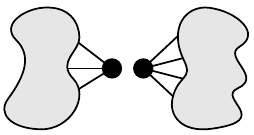}
\hspace{-5mm}
\raisebox{5mm}{
\begin{tabular}{c}
$\xrightarrow[\hspace*{1.6cm}]{\text{identification}}$ \\
$\xleftarrow[\text{cleaving}]{\hspace*{1.6cm}}$
\end{tabular}
\hspace{-5mm}
}
 \labellist
\small\hair 2pt
\pinlabel {$G_1$} at 14 29
\pinlabel {$G_2$} at 52 29
\endlabellist  
\includegraphics[scale=1]{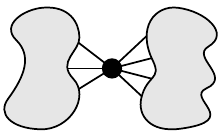}
\qquad
  \labellist
\small\hair 2pt
\pinlabel {$G_1$} at 14 29
\pinlabel {$G_2$} at 52 29
\endlabellist  
\includegraphics[scale=1]{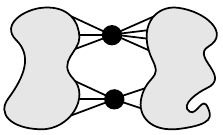}
\raisebox{6mm}{
$\xleftrightarrow[\text{twisting}]{\text{Whitney}}$ 
}
 \labellist
\small\hair 2pt
\pinlabel {$G_1$} at 14 29
\pinlabel {\raisebox{\depth}{\scalebox{1}[-1]{{$G_2$}}}} at 52 9
\endlabellist  
\includegraphics[scale=1]{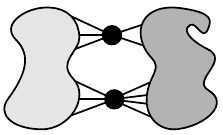}

     \caption{2-isomorphism: vertex identification, vertex cleaving and Whitney twisting}
         \label{ddwe}
\end{figure}

\medskip

The cycle matroid $C(\G)$ of a  graph embedded in a surface, $\G$, records no information about its embedding. Thus the symbiotic relationship above between matroid theory and graph theory does not extend to one between matroid theory and \emph{topological} graph theory.  For a matroidal analogue of a graph embedded in a surface, one should instead consider delta-matroids.
Delta-matroids generalise matroids and were introduced in the mid-1980s, independently, by Bouchet in~\cite{MR904585};  Chandrasekaran and Kabadi, under the name of \emph{pseudo-matroids}, in~\cite{MR959006}; and  Dress and Havel, under the name of \emph{metroids}, in~\cite{MR866162}. (Here we follow the terminology and notation of Bouchet.) 
  In \cite{CMNR1}, Chun, Moffatt, Noble and Rueckriemen proposed that a symbiotic relationship, analogous to that between graph theory and matroid theory,  
   holds between topological graph theory and delta-matroid theory. 
   This perspective has led to a number of recent advances in both areas.

 For this relationship, it is convenient (but not essential) to realise graphs embedded in surfaces as ribbon graphs (these are defined in Section~\ref{srg}).  The delta-matroid $D(\G)$ of a ribbon graph $\G$ provides a topological analogue of a cycle matroid of a graph. 
 Just as in the classical case of graphs and matroids, care must be taken when passing between ribbon graphs and delta-matroids since different ribbon graphs can have the same delta-matroid. In this paper we characterise when two ribbon graphs have the same delta-matroid. That is, we provide an analogue of Whitney's Theorem for ribbon graphs their delta-matroids:
 \begin{theorem}\label{t1}
Let $\G$ and $\bH$ be  ribbon  graphs, and let $D(\G)$ and $D(\bH)$ be their delta-matroids.
Then $D(\G)\iso D(\bH)$ if and only if $\G$ can be obtained from $\bH$ by ribbon graph isomorphism,
 vertex joins, vertex cuts, or
  mutation.
\end{theorem}
Vertex joins, vertex cuts and mutation are described in  Definitions~\ref{dsga} and~\ref{sas}. However, for digesting the theorem statement, the reader is likely to find that a quick look at Figures~\ref{vcj} and~\ref{padon} suffices for the moment.

The result in Theorem~\ref{t1} is likely to agree with the reader's intuition. However, from the perspective of the classical graphs case, the approach to its proof may be unexpected. In particular, Theorem~\ref{t1} is entirely independent of Whitney's 2-isomorphism Theorem.
Instead, its proof relies upon Cunningham's theory of graph decompositions \cite{MR655562}, and  Bouchet's work on circle graphs~\cite{MR918395}. 
Indeed we feel that the proof provides an illuminating example of how ribbon graph theory and delta-matroid theory synchronise by bringing together a number of independently-developed results from each area.

\section{Background}\label{s2}
In this section, we give brief overview of some relevant definitions and  properties of delta-matroids and ribbon graphs. The material presented here is standard, and a reader familiar with ribbon graphic delta-matroids may safely skip or skim this section. 
Additional background on delta-matroids can be found in~\cite{CMNR1,bcc}, and on ribbon graphs in  \cite{MR3086663}.

\subsection{Delta-matroids}\label{s21}

A \emph{set system} is a pair $D=(E, \F)$, where $E$ is a finite set, called the \emph{ground set},  and $\F$ is a collection of subsets of $E$.
The subsets of $E$ in $\F$ are called \emph{feasible sets}. 
A set system $D=(E, \F)$  is \emph{proper} if $\F$ is non-empty,  is  \emph{normal} if $\F$ contains the empty set, and is \emph{even} if every feasible set is of the same parity (i.e., the feasible sets are all of odd size or are all of even size). 

The   \emph{symmetric difference},   $X\triangle Y$, of two sets $X$ and $Y$ is $(X\cup Y) \ba (X\cap Y)$.
A set system  $D=(E, \F)$ is said to satisfy the \emph{Symmetric Exchange Axiom} if for any $X$ and $Y$ in $\F$, if there exists $u\in X\triangle Y$, then there exists $v\in X\triangle Y$ such that $X\triangle \{u, v\}$ is in $\F$.  

 A \emph{delta-matroid}  is a proper set system $D=(E, \F)$ that satisfies the  Symmetric Exchange Axiom. 
 A \emph{matroid} is a delta-matroid in which all feasible sets have the same size, in which case the feasible sets are called \emph{bases}. (This definition is equivalent to the usual basis definition of a matroid.)
 
Two delta-matroids or set systems $D_1=(E_1, \F_1)$ and $D_2=(E_2, \F_2)$ are \emph{isomorphic}, written $D_1\cong D_2$,  if there exists a bijection $\phi:E_1 \rightarrow{}E_2$ such that a subset $X$ of $E_1$ is in $\F_1$ if and only if $\phi(X)$ is in $\F_2$. 

\medskip

We will consider the operations of \emph{twisting} and \emph{loop complementation} on set systems. Twisting was introduced by Bouchet in~\cite{MR904585}, and loop complementation by  Brijder and Hoogeboom in~\cite{MR2838021}.
Let $D=(E, \F)$ be a set system, $e\in E$, and $A\subseteq E$.
The \emph{twist} of $D$ with respect to $A$, denoted by $D\ast A$, is $(E, \{A \triangle X : X\in \F\})$.

\emph{Loop complementation} of $D$ on $e$, denoted by $D+e$, is defined to be the set system $(E,\mathcal{F}')$ where
\[ \mathcal{F}'= \mathcal{F} \triangle \{ F\cup e : F\in \mathcal{F} \text{ and } e\notin F  \} .\]
If $e_1, e_2 \in E$ then $(D+e_1)+e_2 = (D+e_2)+e_1 $. This means that if $A=\{a_1, \ldots , a_n\}\subseteq E$ we can unambiguously define the {\em loop complementation} of $D$ on $A$, by $D+A:= D+a_1+\cdots + a_n$.

If $D$ is a delta-matroid, then $D\ast A$ is always a delta-matroid, but, in general, $D+A$ need not be. However, for the class of delta-matroids we are interested in here (the delta-matroids of ribbon graphs), it always is (this follows from Item~\eqref{dan4} of Theorem~\ref{dan} below).

\medskip

We shall make use of properties of binary delta-matroids.  
Let  $\mat{A}$ be a symmetric matrix over some field, whose rows and columns are labelled (in the same order) by a set $E$. 
For $X\subseteq E$,  let $\mat{A}[X]$ denote the principal submatrix of $\mat{A}$ given by the rows and columns indexed by $X$. 
Define a collection $\F$ of subsets of $E$ by  taking $X \in \F$   if and only if  $\mat{A}[X]$  is non-singular. 
By convention,  $\mat{A}[\emptyset]$ is considered to be non-singular. 
Bouchet proved in~\cite{ab88} that $D(\mat{A}):= (E,\F)$ is a delta-matroid.  Note that $D(\mat{A})$ is necessarily normal.
 
A normal delta-matroid $D$ is said to be \emph{binary} if $D\iso D(\mat{A})$ for some symmetric matrix $\mat{A}$  over $\GF(2)$.
A delta-matroid  is  \emph{binary} if it is a twist of a normal binary delta-matroid.

An important fact for us here is that  normal binary delta-matroids are completely determined by their feasible sets of size at most 2, a result due to  Bouchet and Duchamp~\cite{BD91}:
\begin{theorem}\label{ads}%[Bouchet and Duchamp \cite{BD91}]
Let $(E,\F')$ be a normal set system. 
Then there is exactly one binary delta-matroid $D=(E,\F)$ such that the collections of sets of size at most two in $\F$ and $\F'$ are identical.
\end{theorem}

The \emph{fundamental graph} of a normal even delta-matroid $D=(E,\F)$ is the graph with vertex set $E$, and with an edge $xy$ if and only if $\{x,y\}\in \F$. The fundamental graph completely determines the ground set and the feasible sets of size at most two of an even set system. Hence, by Theorem~\ref{ads}, it uniquely determines a  normal even binary delta-matroid. Thus we have that a normal even binary delta-matroid is completely determined by its fundamental graph, and, conversely, that every simple graph uniquely determines a  normal even binary delta-matroid.

\subsection{Ribbon graphs}\label{srg}

A {\em ribbon graph} $\G=\left(V,E\right)$ is a surface with boundary, represented as the union of two sets of discs --- a set $V$ of {\em vertices} and a set $E$ of {\em edges} --- such that: (1) the vertices and edges intersect in disjoint line segments; (2) each such line segment lies on the boundary of precisely one vertex and precisely one edge; and (3) every edge contains exactly two such line segments.  A ribbon graph is shown in Figure~\ref{ar2}.

A \emph{bouquet} is a ribbon graph on exactly one vertex, and a \emph{quasi-tree} is a ribbon graph that has exactly one boundary component (recalling that a ribbon graph is always a surface with boundary).  A loop in a ribbon graph is  \emph{non-orientable} if together with its incident vertex it forms a M\"obius band, and is  \emph{orientable} otherwise.  For a ribbon graph $\G=(V,E)$, and a subset $A$ of its edges, we let $k(A)$ denote the number of connected components of the ribbon subgraph $(V,A)$, and we let $b(A)$ denote the number of boundary components of $(V,A)$.

It is well-known that ribbon graphs are equivalent to cellularly embedded graphs in surfaces. (Ribbon graphs arise naturally from neighbourhoods of cellularly embedded graphs. On the other hand, topologically a ribbon graph is a  surface with boundary, and capping the holes gives rise to a cellularly embedded graph in the obvious way. See \cite{MR3086663,GT87} for details.)  
We say that a ribbon graph is \emph{plane} if it  describes a graph  cellularly embedded in a disjoint union of spheres. 

Ribbon graph equivalence agrees with the usual notion of equivalence of graphs in surfaces. 
Two ribbon graphs  $\G_1=(V_1, E_1)$ and $\G_2=(V_2, E_2)$ are \emph{equivalent} if there is a homeomorphism (which is orientation preserving when $\G_1$ is orientable) from $\G_1$ to $\G_2$  mapping $V_1$ to $V_2$, and $E_1$ to $E_2$, and so preserving the cyclic order of half-edges at each vertex.

Arrow presentations, from~\cite{Ch09}, provide  convenient combinatorial descriptions of  ribbon graphs. An \emph{arrow presentation}  is a set of closed curves,  each with a collection of disjoint  labelled arrows  lying on them, and  where each label appears on precisely two arrows. An arrow presentation is shown in Figure~\ref{ar1}.

Arrow presentations describe  ribbon graphs.
A ribbon graph $\G$ can be formed from an arrow presentation by  identifying each closed curve with the boundary of a disc (forming the vertex set of $\G$). Then, for each pair of $e$-labelled arrows, taking a disc (which will form an edge of $\G$), orienting its boundary, placing two disjoint arrows on its boundary that point in the direction of the orientation, and identifying each $e$-labelled arrow on this edge. See Figure~\ref{ar}.
 Conversely a ribbon graph can be described as an arrow presentation by arbitrarily labelling and orienting the boundary of each edge disc of $\G$. Then on each arc where an edge disc intersects a vertex disc, place an arrow on the vertex disc, labelling the arrow with the label of the edge it meets and directing it consistently with the orientation of the edge disc boundary. The boundaries of the vertex set marked with these labelled arrows give  an arrow presentation.  
Arrow presentations are equivalent if they describe  isomorphic ribbon graphs.

\begin{figure}[!ht]
       \begin{subfigure}[t]{0.4\textwidth}
    \centering
         \labellist
\small\hair 2pt
\pinlabel {$1$} at 180 104
\pinlabel {$2$} at 17 17
\pinlabel {$3$} at 120 120
\pinlabel {$4$} at  154 24
\endlabellist
\includegraphics[scale=0.5]{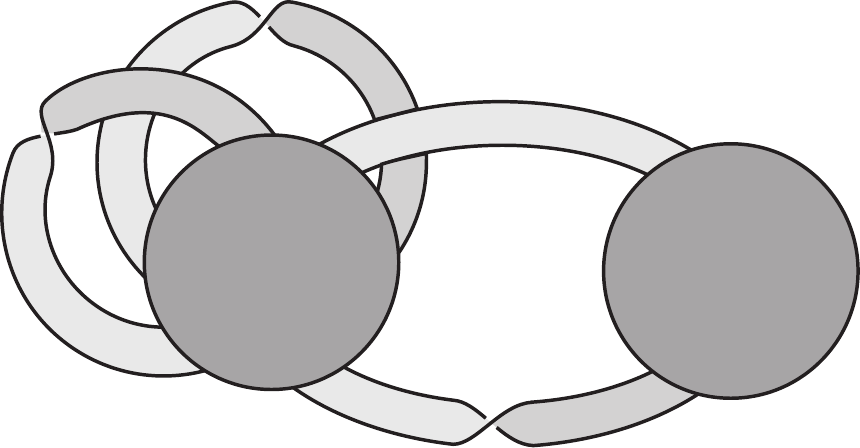}
        \caption{A ribbon graph}
        \label{ar2}
    \end{subfigure}
    \qquad
     \centering
    \begin{subfigure}[t]{0.4\textwidth}
    \centering
   \labellist
\small\hair 2pt
\pinlabel {$1$} at 159 97
\pinlabel {$2$} at 16 26
\pinlabel {$3$} at 10 74
\pinlabel {$4$} at  156 17
\pinlabel {$1$} at 77 103
\pinlabel {$2$} at 40 104
\pinlabel {$3$} at 102 82
\pinlabel {$4$} at  83 17
\endlabellist
\includegraphics[scale=0.5]{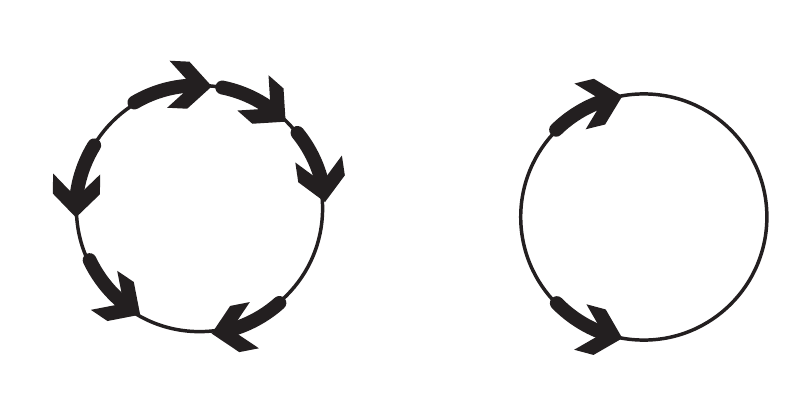}
        \caption{An arrow presentation}
\label{ar1}
    \end{subfigure}
  \qquad

    \caption{A ribbon graph and its description as an arrow presentation}
        \label{ar}
\end{figure}
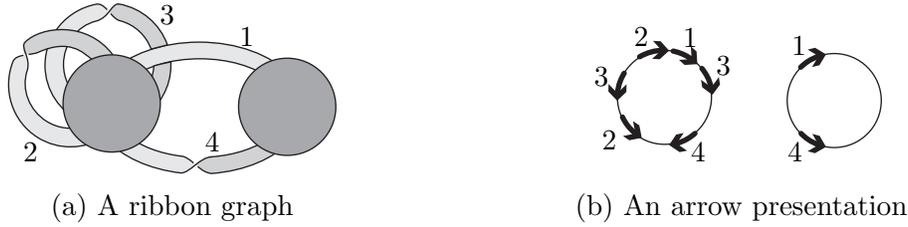

We shall make use of two notions of duals of ribbon graphs: partial duals and partial petrials. Both notions have a natural geometric definition in terms of ribbon graphs, but for the application here, their definitions in terms of arrow presentations are most helpful.

Informally, partial petriality, introduced in \cite{MR2869185}, may be though of as an operation that ``adds or removes half-twists to the edges of a ribbon graph''. 
     Formally, the \emph{partial petrial} of $\G$ formed with respect to $e$ is the ribbon graph $\G^{\tau(e)}$ obtained from $\G$ by  detaching an end of $e$ from its incident vertex $v$ creating arcs $[a,b]$ on $v$, and $[a',b']$ on $e$  (so that $\G$ is recovered by identifying $[a,b]$ with $[a',b']$), then reattaching the end  by identifying the arcs antipodally, so that $[a,b]$ is identified with $[b',a']$. For a subset of edges $A$ of $\G$, $\G^{\tau(A)}$ is the ribbon graph obtained by forming the  partial petrial of with respect to each edge in $A$ (in any order).
 
 In terms of arrow presentations, an arrow presentation for $\G^{\tau(e)}$ can be obtained by reversing the direction of exactly one $e$-labelled arrow in an arrow presentation for $\G$.  See Figure~\ref{du}.

\begin{figure}[!ht]
    \centering
    \labellist
\small\hair 2pt
\pinlabel {$e$} at 45 35 
\pinlabel {$e$} at 71 35
\endlabellist  
\includegraphics[scale=0.5]{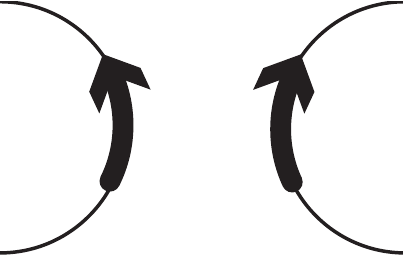}
\qquad\raisebox{6mm}{$\xleftrightarrow[\hspace*{1cm}]{\G^{\tau(e)}}$}\qquad
            \labellist
\small\hair 2pt
\pinlabel {$e$} at 45 35 
\pinlabel {$e$} at 71 35
\endlabellist  
\includegraphics[scale=0.5]{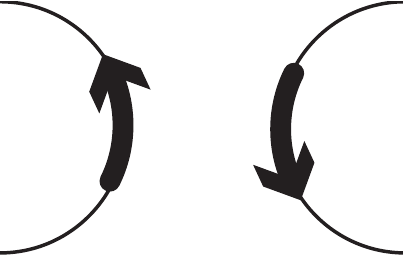}
\qquad\raisebox{6mm}{$\xleftrightarrow[\hspace*{1cm}]{\G^{\{e\}}}$}\qquad
 \labellist
\small\hair 2pt
\pinlabel {$e$} at 60 67 
\pinlabel {$e$} at 60 6
\endlabellist  
\includegraphics[scale=0.5]{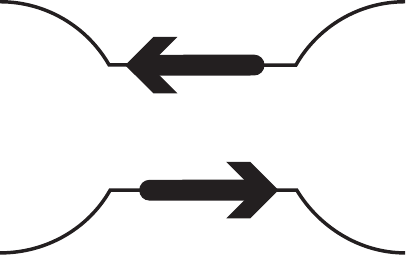}
     \caption{Partial petrial and  partial dual in terms of arrow presentations}
         \label{du}
\end{figure}
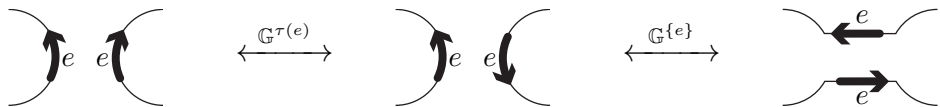

In terms of ribbon graphs, the \emph{(geometric) dual} $\G^*$ of $\G$, is formed by gluing a disc, which will form a vertex of $\G^*$, to each boundary component of $\G$ by identifying the boundary of the disc with the curve, then removing the interior of all vertices of $\G$. A \emph{partial dual}, introduced by Chmutov in~\cite{Ch09}, is obtained by ``forming the dual with respect to a subset of the edges'' as follows.  Let $\G=(V, E)$ be a ribbon graph, $A\subseteq E$, and regard the boundary components of the  ribbon subgraph $(V,A)$  as curves on $\G$. Glue a disc to $\G$ along each of these curves by identifying the boundary of the disc with the curve, and remove the interior of all vertices of $\G$. The resulting ribbon graph is  the partial dual $\G^{A}$. 

Partial duality can be described as a splicing operation on arrow presentations.  Let $\G$ be a ribbon graph with an edge $e$.  Then the \emph{partial dual} of $\G$ with respect to $e$ is the ribbon graph  denoted  $\G^{\{e\}}$  obtained from $\G$ by the following process: (1) describe $\G$ as an arrow presentation, (2) `splice' the arrow presentation at the two $e$-labelled arrows as indicated in the right-hand side of Figure~\ref{du}.  (3) The ribbon graph described by this arrow presentation is $\G^{\{e\}}$.

When $e\neq f$ are edges of a ribbon graph $\G=(V,E)$, it is easily seen that $(\G^{\{e\}})^{\{f\}}=(\G^{\{f\}})^{\{e\}}$. Thus  $\G^A$ can be obtained from $\G$ by forming the partial dual with respect to each edge of $A$ in any order.

\subsection{Ribbon graphic delta-matroids}\label{s22}

The \emph{cycle matroid} $C(G)$ of a graph $G=(V,E)$ can be defined as the pair  $(E, \B)$, where $\B$ consists of the edge sets of the maximal spanning forests of $G$. Thus the cycle matroid  $C(\G)$ of a ribbon graph  $\G=(V,E)$ can be defined as $(E, \B)$, where 
$\B :=  \{ A\subseteq E : b(A)=k(E) \text{ and } (V,A) \text{ is a plane ribbon graph}\} $ (since a connected, plane ribbon graph with exactly one boundary component is precisely a tree).  The delta-matroid $D(\G)$ of a ribbon graph $\G$ arises by dropping the plane requirement in this formulation of $C(\G)$. It provides an embedded graph analogue of a cycle matroid.  

Let $\G=(V, E)$ be a ribbon graph. Then $D(\G)$ is the set system $(E, \F)$ where $\F := \{ A\subseteq E : b(A)=k(E)\} $. In particular, this means that when $\G$ is connected, $\F$ is collection of the edge sets of all spanning quasi-trees of $\G$ (whereas $C(\G)$ is formed from the edge sets of the spanning trees).
As an example, if  $\G$ is  the ribbon graph in Figure~\ref{ar2}, then 
 $D(\G)=(E,\F)$ where $E=\{1,2,3,4\}$ and 	
$ \F=\{ \{1\}, \{4\}, \{1,2\} , \{1,3\} , \{1,4\} , \{2,4\} , \{3,4\}   ,\{1,2,4\}  ,\{1,2,3,4\} \}$.  

A proof that $D(\G)$ is indeed a delta-matroid can be found in \cite{ab2}, where it is written in terms of Eulerian circuits in medial graphs, or  in~\cite{CMNR1}, where it is written in terms of ribbon graphs.

\medskip

The importance of partial duals and partial petrials here is that they are the ribbon graph analogues of the fundamental delta-matroid operations of twisting and loop-complementation. The following theorem collects this together with some other  properties of ribbon graphic delta-matroids.
 In the theorem, Items~\ref{dan1} and~\ref{dan2} are due to Bouchet and from~\cite{ab2} and~\cite{ab88}, respectively. Items~\ref{dan3} and~\ref{dan4} are due to Chun, Moffatt, Noble and Rueckriemen and from~\cite{CMNR1} and~\cite{CMNR2}, respectively.
\begin{theorem}\label{dan}
Let $\G=(V, E)$ be a ribbon graph, and $A\subseteq E$. Then, 
\begin{enumerate}
\item \label{dan1} $D(\G)$ is even if and only if $\G$ is orientable, 
\item \label{dan2} $D(\G)$ is binary, 
\item\label{dan3} $D(\G^A) = D(\G)\ast A$,
\item   \label{dan4} $D(\G^{\tau(A)})=D(\G)+A$.
\end{enumerate}
\end{theorem}

\section{Mutation and the main result}

This section provides the terminology for Theorem~\ref{t1}. In particular, it introduces ribbon graph analogues of  vertex identification, vertex cleaving, and  Whitney twisting.
We start with the analogues of vertex identification and vertex cleaving. 
\begin{definition}\label{dsga}
Suppose that $\G_1$ and $\G_2$ are ribbon graphs.  For $i=1,2$, suppose that $\alpha_i$ is an  arc that lies on the boundary of $\G_i$  and entirely on  a vertex boundary.  If a ribbon graph $\G$ can be obtained from $\G_1$ and $\G_2$ by identifying the arc $\alpha_1$ with $\alpha_2$ (where the identification merges the vertices), then we say that 
$\G$ is obtained from $\G_1$ and $\G_2$ by a \emph{vertex join}, and that 
 $\G_1$ and $\G_2$  are obtained from $\G$ by a \emph{vertex cut}.
\end{definition}
The operations introduced in Definition~\ref{dsga}  are illustrated in Figure~\ref{vcj}, and are standard operations in ribbon graph theory. 
We say that two ribbon graphs are related by vertex joins and vertex cuts if there is a sequence of such operations taking one to the other. Note that vertex joins and vertex cuts provide a way to delete or add isolated vertices.
  It is important to observe that the definition of a vertex join does not allow for any ``interlacing'' of the edges of $G_1$ and $G_2$. For example, the ribbon graph in Figure~\ref{ar2} cannot be written as the vertex join of two non-trivial ribbon graphs.

\begin{figure}[!ht]
    \centering
     \labellist
\small\hair 2pt
\pinlabel {$\G_1$} at 18 31
\pinlabel {$\G_2$} at 113 31
\endlabellist  
\includegraphics[scale=1]{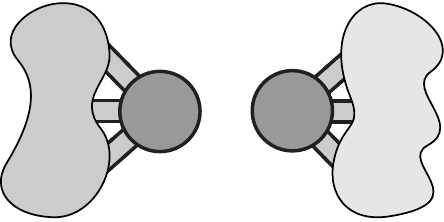}
\quad\raisebox{10mm}{
\begin{tabular}{c}
$\xrightarrow[\hspace*{1cm}]{\text{vertex join}}$ \\
$\xleftarrow[\text{vertex cut}]{\hspace*{1cm}}$
\end{tabular}
}
 \quad
 \labellist
\small\hair 2pt
\pinlabel {$\G$} at 46 31
\endlabellist  
\includegraphics[scale=1]{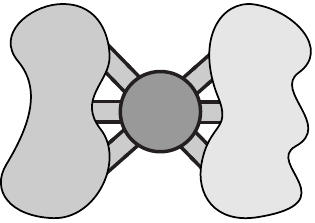}
     \caption{Vertex joins and vertex cuts}
         \label{vcj}
\end{figure}

The following definition introduces an operation called mutation which  provides a ribbon graph analogue of Whitney twisting. It is illustrated in Figure~\ref{padon}.
The figure illustrates a local change in a ribbon graph (so the ribbon graphs are identical outside of the region shown), and the two parts of vertices that are shown in it may be from the same vertex.

\begin{figure}[!ht]
  
    \centering
     \labellist
\small\hair 2pt
\pinlabel {$\G_2$} at 20 33
\endlabellist 
        \includegraphics[scale=1]{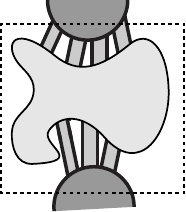}
        \qquad
                     \labellist
\small\hair 2pt
\pinlabel {\reflectbox{$\G_2$}} at 32 33
\endlabellist 
            \includegraphics[scale=1]{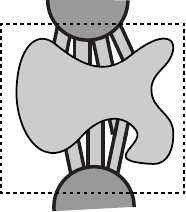}
      \qquad
        \labellist
\small\hair 2pt
\pinlabel {\raisebox{\depth}{\scalebox{1}[-1]{{$\G_2$}}}} at 20 29
\endlabellist 
            \includegraphics[scale=1]{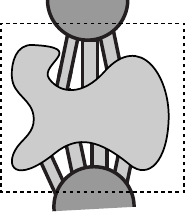}           
           \qquad
            \labellist
\small\hair 2pt
\pinlabel {\raisebox{\depth}{\scalebox{1}[-1]{{\reflectbox{$\G_2$}}}}} at 32 29
\endlabellist 
                        \includegraphics[scale=1]{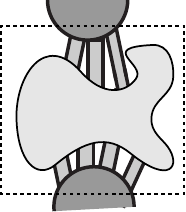}

    \caption{Mutation for ribbon graphs. (The vertex parts shown need not be from distinct vertices.)}
        \label{padon}
\end{figure}

\begin{definition}\label{sas}
Let $\G_1$ and $\G_2$ be ribbon graphs. For $i=1,2$, let $\alpha_i$ and $\beta_i$ be two disjoint directed arcs that lie on the boundary of $\G_i$  and lie entirely on boundaries of (one or two) vertices. Furthermore suppose that $\G$ is a ribbon graph that is obtained by identifying the arcs $\alpha_1$ with $\alpha_2$, and $\beta_1$ with $\beta_2$, where both identifications are consistent with the direction of the arcs. (The identification merges the vertices.) Suppose further that $\bH$  is a ribbon graph obtained by either:
\begin{enumerate}
 \item identifying $\alpha_1$ with $\alpha_2$, and  $\beta_1$ with $\beta_2$, where the identifications are \emph{inconsistent} with the direction of the arcs;
\item identifying $\alpha_1$ with $\beta_2$, and  $\beta_1$ with $\alpha_2$, where the identifications are \emph{consistent} with the direction of the arcs;
\item identifying $\alpha_1$ with $\beta_2$, and  $\beta_1$ with $\alpha_2$, where the identifications are \emph{inconsistent} with the direction of the arcs.
 \end{enumerate}
Then we say that $\G$ and $\bH$ are related by \emph{mutation}, and that $\G$ and $\bH$ are \emph{mutants}.
\end{definition}

We note that implicit in Definition~\ref{sas} is that the arcs $\alpha_i$ and $\beta_i$ lie on at least three distinct vertices. (Since if they were on a total of two vertices, $\G$ would not be a ribbon graph.) It may also be the case that the vertices of $\G_1$ and $\G_2$ are merged into a single vertex (so, for example, mutation can act on bouquets).

\begin{lemma}\label{mkdjk}
Let $\G$ and $\bH$ be ribbon graphs on the same edge set $E$, and let $A\subseteq E$. Then the following hold:
\begin{enumerate}
\item $\G$ and $\bH$ are mutants if and only if $\G^{\tau (A)}$ and $\bH^{\tau (A)}$ are mutants,
\item $\G$ and $\bH$ are mutants if and only if $\G^{A}$ and $\bH^{A}$ are  mutants.
\end{enumerate}
\end{lemma}
\begin{proof}
The results follow since mutation, partial duality, and partial petriality are local operations acting on different parts of the ribbon graph, and so may be carried out in any order. This can be seen easily by considering arrow presentations. Suppose that $\G$,  $\G_1$,  $\G_2$, and $\bH$ are all as in  Definition~\ref{sas}. An arrow presentation for $\G_1 \sqcup \G_2$ can drawn as in Figure~\ref{adhk}, where all the arrows are contained inside the shaded regions  in the figure, and where $\G_1$ is represented by the arrow presentation $\calA_1$, and $\G_2$ by $\calA_2$. In particular, this means that corresponding pairs of arrows are contained in the same shaded regions.  $\G$ and $\bH$ are then obtained by replacing the configuration shown in the dotted box with one of those shown in Figure~\ref{anjda}. Then as partial duality and partial petriality act entirely within the shaded regions of $\calA_1$ and $\calA_2$, it follows that mutation commutes with each of partial duality and partial petriality. The result then follows since partial petriality and partial duality are involutory.
\end{proof}

\begin{figure}[!ht]
    \centering
    \begin{subfigure}[t]{0.4\textwidth}
    \centering
    \labellist
\small\hair 2pt
\pinlabel $\calA_1$ at 26 35
\pinlabel $\calA_2$ at 100 35
\endlabellist
        \includegraphics[scale=1]{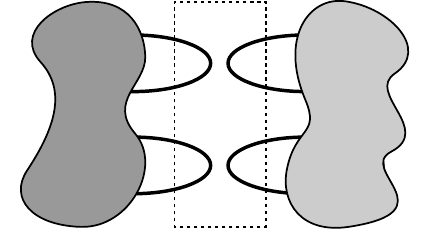}
        \caption{Arrow presentation}
\label{adhk}
    \end{subfigure}
  \qquad
    \begin{subfigure}[t]{0.4\textwidth}
    \centering
        \includegraphics[scale=1]{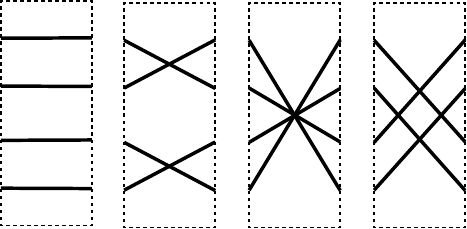}
        \caption{Configurations for mutation}
        \label{anjda}
    \end{subfigure}
    \caption{Mutation in terms of arrow presentations}
        \label{anjd}
\end{figure}

\section{The proof of Theorem~\ref{t1}}\label{s3}

Our proof of of Theorem~\ref{t1} depends upon a result about circle graphs.  
A \emph{chord diagram} consists of a circle in the plane and a number straight line segments, called \emph{chords}, whose end-points lie on the circle. The end-points of chords should all be distinct. Chord diagrams can be regarded as \emph{double occurrence words}, which are cyclically ordered finite sequences of letters over some alphabet, such that each letter occurs exactly twice in the sequence. (Assign a letter to each chord, then the order of the ends of the chord on the circle determines the sequence.) Two   double occurrence words are \emph{equivalent} if there is a bijection between their alphabets that sends one double occurrence word to the other or to the reversed word. Two chord diagrams are \emph{isomorphic} if they give rise to equivalent  double occurrence words. 

The \emph{intersection graph} of a chord diagram is the graph $G=(V,E)$ where $V$ is the set of chords, and where $uv\in E$ if and only if the chords $u$ and $v$ intersect. A graph is a \emph{circle graph} if it is the intersection graph of a chord diagram.  
In general, different chord diagrams can have the same intersection graph.

  Following the terminology of  \cite{CL}, a  \emph{share} in a chord diagram  is a pair of disjoint arcs of its circle that have the property that if an end of a chord lies on the arcs, then the other end of the chord also lies on the arcs. In terms of double occurrence words, a \emph{share} $w$ consists of two disjoint subwords such that if one letter is contained in the pair of subwords, then so is the other. 
  
    Let $w=w_1w_2w_3w_4$ be a double occurrence word in which $w_2$ and $w_4$ form a share (and therefore so do $w_1$ and $w_3$).  Then the words $w_1\overline{w_2}w_3\overline{w_4}$, $w_1{w_4}w_3{w_2}$,  and $w_1\overline{w_4}w_3\overline{w_2}$ are \emph{mutants} of $w$, where $\overline{w_i}$ denotes the word obtained by reversing the order of $w_i$.  Two chord diagrams are said to be related by \emph{mutation} if they define mutant double occurrence words (up to equivalence).
 Figure~\ref{dfg} indicates how mutation acts on a chord diagram.

\begin{figure}[!ht]
  
    \centering
     \labellist
\small\hair 2pt
\pinlabel {\rotatebox{90}{share}} at 21 20
\endlabellist 
        \includegraphics[scale=1]{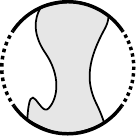}
        \qquad
                     \labellist
\small\hair 2pt
\pinlabel {\reflectbox{\rotatebox{90}{share}}} at 19 20
\endlabellist 
            \includegraphics[scale=1]{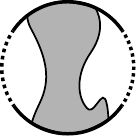}
      \qquad
        \labellist
\small\hair 2pt
\pinlabel {\raisebox{\depth}{\scalebox{1}[-1]{{\rotatebox{90}{share}}}}} at 21 20
\endlabellist 
            \includegraphics[scale=1]{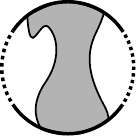}           
           \qquad
            \labellist
\small\hair 2pt
\pinlabel {\raisebox{\depth}{\scalebox{1}[-1]{{\reflectbox{\rotatebox{90}{share}}}}}} at 19 20
\endlabellist 
                        \includegraphics[scale=1]{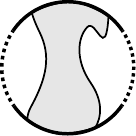}

    \caption{Mutation for chord diagrams}
        \label{dfg}
\end{figure}

The following theorem determines when chord diagrams have the same intersection graph. 
It appears implicitly in  \cite{MR918395,MR2437322,MR1072233}, and explicitly in~\cite{CL} (whose terminology we follow here).
\begin{theorem}\label{t3}
Two chord diagrams are related by mutation and isomorphism if and only if they have isomorphic intersection graphs.
\end{theorem}
The idea behind the proof of this theorem is to use Cunningham's theory of graph decompositions from~\cite{MR655562} to obtain a decomposition of an intersection graph into `prime' graphs that have unique intersection graphs. This gives a decomposition of a chord diagram into minimal shares. Mutation then corresponds to the choices that are made when reassembling a chord diagram from these shares.

An orientable bouquet is equivalent to a chord diagram. (The boundary of the vertex forms the circle of the chord diagram, and the edges determine the chords). 
Furthermore, it is easy to see that  two orientable bouquets are mutants if and only if their chord diagrams are mutants. If we define  the intersection graph of a orientable bouquet to be the intersection graph of the corresponding chord diagram, then Theorem~\ref{t3} immediately gives the following. 
\begin{corollary}\label{anja}
Two orientable bouquets are related by mutation and isomorphism if and only if they have isomorphic intersection graphs.
\end{corollary}

One further observation we use is that the intersection graph of an  orientable bouquet $\G$ coincides with the fundamental graph of $D(\G)$ (since $\{e,f\}$ is a feasible set of $D(\G)$  if and only if the chords corresponding to edges $e$ and $f$ intersect). 

Combining these observations gives the orientable case of the following result.

\begin{lemma}\label{l1}
Let $\G$ and $\bH$ be  bouquets.
Then  $D(\G)\iso D(\bH)$ if and only if $\G$ and $\bH$ are related by mutation and isomorphism. 
\end{lemma}
\begin{proof}
If one of $\G$ or $\bH$ is orientable and the other non-orientable, then they cannot be related by mutation or isomorphism, and, by Item~\ref{dan1} of Theorem~\ref{dan}, they cannot have isomorphic delta-matroids. Thus, without loss of generality we may assume that  $\G$ and $\bH$ are both orientable or both non-orientable.

First suppose that $\G$ and $\bH$ are both orientable. Then $D(\G)$ and $D(\bH)$ are normal, even, binary delta-matroids, and are therefore completely determine by their fundamental graphs. 
Thus $D(\G)\iso D(\bH)$ if and only if they have isomorphic fundamental graphs. 
This happens if and only if $\G$ and $\bH$ have isomorphic intersection graphs (by the observation appearing just after Corollary~\ref{anja}), and hence, by Corollary~\ref{anja}, if and only if $\G$ and $\bH$ are related by mutation and isomorphism. 

Now suppose that $\G$ and $\bH$ are both non-orientable. 
Let $A_{\G}\subseteq E(\G)$ denote the set of all non-orientable loops in $\G$, and $A_{\bH}\subseteq E(\bH)$ denote the set of all non-orientable loops in $\bH$. 
Then $\{e\}$ is a feasible set of $D(\G)$ if and only if $e\in A_{\G}$, and  $\{e\}$ is a feasible set of $D(\bH)$ if and only if $e\in A_{\bH}$. 
Thus any isomorphism between either $\G$ and $\bH$, or $D(\G)$ and $D(\bH)$ identifies $A_{\G}$ and $A_{\bH}$.
Because of this identification, we see 
 $D(\G)\iso D(\bH)$ if and only if  $D(\G)+  A_{\G} \iso D(\bH)+ A_{\bH}$. By Item~\ref{dan4} of Theorem~\ref{dan}, this happens if and only if  $D(\G^{\tau(  A_{\G})}) \iso D(\bH^{\tau( A_{\bH})})$. 
 However, since $\G^{\tau( A_{\G})}$ and  $\bH^{\tau( A_{\bH})}$ are both orientable, it follows from the above that $D(\G^{\tau(  A_{\G})}) \iso D(\bH^{\tau( A_{\bH})})$ if and only if $\G^{\tau( A_{\G})}$ and  $\bH^{\tau( A_{\bH})}$  are related by mutation and isomorphism, which, by Lemma~\ref{mkdjk} happens if and only if $\G$ and  $\bH$  are related by mutation and isomorphism.
\end{proof}

\begin{lemma}\label{dsad}
Let $\G$ and $\bH$ be connected ribbon graphs. Then  $D(\G)\iso D(\bH)$ if and only if $\G$ and $\bH$ are related by mutation and isomorphism. 
\end{lemma}

\begin{proof}
If  $D(\G)\iso D(\bH)$, then any feasible set $X_{\G}$ of $D(\G)$ is identified with a feasible set  $X_{\bH}$  of $ D(\bH)$.
Then $D(\G)\iso D(\bH)$ if and only if $D(\G)\ast X_{\G} \iso D(\bH)\ast X_{\bH}$, for corresponding  feasible sets   $X_{\G}$ and $X_{\bH}$.
By Item~\ref{dan3} of Theorem~\ref{dan}, this holds if and only if $D(\G^{X_{\G}}) \iso D(\bH^{X_{\bH}})$, where $X_{\G}$ and $X_{\bH}$ are edge sets of corresponding quasi-trees of $\G$ and $\bH$. As both $\G^{X_{\G}}$ and  $\bH^{X_{\bH}}$ are 1-vertex ribbon graphs (as we have formed the partial dual with respect to a quasi-tree), by Lemma~\ref{l1}, this happens if and only if $\G^{X_{\G}}$ and $\bH^{X_{\bH}}$ are related by isomorphism and mutation. 
By Lemma~\ref{mkdjk}, $\G^{X_{\G}}$ and $\bH^{X_{\bH}}$ are related by isomorphism and mutation if and only if $\G$ and $\bH$ are related by isomorphism and mutation, where $X_{\G}$ and  $X_{\bH}$ are corresponding edge sets. 
\end{proof}

The main theorem readily follows.
\begin{proof}[Proof of Theorem~\ref{t1}]
It is not hard to see that vertex joins and vertex cuts do not change the delta-matroid of a ribbon graph. The delta-matroid of a disconnected ribbon graph is the direct sum of the delta-matroids of its components, and so it follows from Lemma~\ref{dsad} that mutation does not change the delta-matroid. 

Conversely, if $D(\G)\iso D(\bH)$, by using vertex joins, we can obtain connected ribbon graphs $\G'$ and $\bH'$  such that $D(\G)\iso D(\G')$ and $D(\bH)\iso D(\bH')$. By Lemma~\ref{dsad} it follows that $\G'$ and $\bH'$ are related by mutation and isomorphism, and so $\G$ and $\bH$ are related by vertex joins,  mutation, isomorphism, and vertex cuts.
\end{proof}

We conclude by noting  three corollaries to Theorem~\ref{t1}.  The first is the observation that the planar case of Whitney's 2-Isomorphism Theorem follows from it. This is since the delta-matroid of a plane ribbon graph is its cycle matroid.
\begin{corollary}
Let $G$ and $H$ planar graphs.
 Then $C(G)$ and  $C(H)$ are isomorphic  if and only if $G$ and $H$ are related by isomorphism, vertex identification, vertex cleaving, or Whitney twisting.
 \end{corollary}

For the second corollary, two graphs are said to be \emph{Tutte equivalent} if they have the same Tutte polynomials.  A useful result in the area of graph polynomials is that the Tutte polynomial is invariant under Whitney flips. This follows from  Whitney's 2-Isomorphism Theorem since the Tutte polynomial of a graph is determined by its cycle matroid. Recently, much attention has been paid to versions of the Tutte polynomial for graphs in surfaces. Many of these topological graph polynomials are determined by delta-matroids (see \cite{CMNR2,CMNR1,KMT,MS,MR3361421}) and so are invariant under mutation by  Theorem~\ref{t1}. 
\begin{corollary}\label{c3}
Connected mutant ribbon graphs have the same 
Bollob\'as-Riordan,  Krushkal, Las~Vergnas, Penrose, ribbon graph and topological transition polynomials.  
\end{corollary}

The \emph{Jones polynomial} of a knot or link is a specialization of the Bollob\'as--Riordan polynomial of all-A ribbon graph obtained from one of its diagrams~\cite{DFKLS}. 
Vertex joins and mutation on ribbon graphs corresponds to connected sums and mutation of link diagrams. Hence,  by Corollary~\ref{c3}, we have the following result. 

\begin{corollary}\label{c4}
The Jones polynomial $V_K(t)$ of a knot $K$ satisfies the following properties.
\begin{enumerate}
    \item  $V_{K_1\#K_2}(t)=V_{K_1}(t)V_{K_2}(t)$, where $K_1\#K_2$ is a connected sum of knots $K_1$ and $K_2$.
    \item $V_K(t)=V_{K'}(t)$, where $K'$ is a mutant of $K$.
\end{enumerate}
\end{corollary}
 This result is a standard fact in knot theory. What is notable here is that for alternating links the result can be deduced from Whitney's Theorem and that the Jones polynomial for alternating links can be recovered from the Tutte polynomial. The above argument is the extension of this classical approach to the case of non-alternating links.


\begin{thebibliography}{10}

\bibitem{MR904585}
A.~Bouchet,
\newblock Greedy algorithm and symmetric matroids,
\newblock Math. Programming, 38  (1987) 147--159.

\bibitem{MR918395}
A.~Bouchet,
\newblock Reducing prime graphs and recognizing circle graphs,
\newblock Combinatorica, 7  (1987) 243--254.

\bibitem{ab88}
A.~Bouchet,
\newblock Representability of {$\triangle$}-matroids,
\newblock In Combinatorics ({E}ger, 1987), volume~52 of Colloq.
  Math. Soc. J{\'a}nos Bolyai, pages 167--182. North-Holland, Amsterdam, 1988.

\bibitem{ab2}
A.~Bouchet,
\newblock Maps and {$\triangle$}-matroids,
\newblock Discrete Math., 78  (1989) 59--71.

\bibitem{BD91}
A.~Bouchet, A.~Duchamp,
\newblock Representability of {$\triangle$}-matroids over {${\rm GF}(2)$},
\newblock Linear Algebra Appl., 146 (1991) 67--78.

\bibitem{MR2838021}
R.~Brijder, H.~J. Hoogeboom,
\newblock The group structure of pivot and loop complementation on graphs and
  set systems,
\newblock European J. Combin., 32  (2011) 1353--1367.

\bibitem{MR959006}
R.~Chandrasekaran, S.~N. Kabadi,
\newblock Pseudomatroids,
\newblock Discrete Math., 71  (1988) 205--217.

\bibitem{Ch09}
S.~Chmutov,
\newblock Generalized duality for graphs on surfaces and the signed
  {B}ollob{\'a}s-{R}iordan polynomial,
\newblock J. Combin. Theory Ser. B, 99  (2009) 617--638.

\bibitem{CL}
S.~Chmutov, S.~Lando,
\newblock Mutant knots and intersection graphs,
\newblock Algebr. Geom. Topol., 7  (2007) 1579--1598.

\bibitem{CMNR2}
C.~Chun, I.~Moffatt, S.~D. Noble, R.~Rueckriemen,
\newblock On the interplay between embedded graphs and delta-matroids,
\newblock Proc. London Math. Soc., 18 (2019) 675--700.

\bibitem{CMNR1}
C.~Chun, I.~Moffatt, S.~D. Noble, R.~Rueckriemen,
\newblock Matroids, delta-matroids and embedded graphs,
\newblock J. Combin. Theory Ser. A, 67 (2019) 7--59.

\bibitem{MR2437322}
B.~Courcelle,
\newblock Circle graphs and monadic second-order logic,
\newblock J. Appl. Log., 6  (2008) 416--442.

\bibitem{MR655562}
W.~H. Cunningham,
\newblock Decomposition of directed graphs,
\newblock SIAM J. Algebraic Discrete Methods, 3  (1982) 214--228.

\bibitem{DFKLS}
O.~T. Dashbach, D.~Futer, E.~Kalfagianni, X.-S. Lin, N.~W. Stoltzfus,
\newblock The Jones polynomial and graphs on surfaces,
\newblock J. Combin. Theory Ser. B, 98  (2008) 384--399. 

\bibitem{MR866162}
A.~Dress, T.~F. Havel,
\newblock Some combinatorial properties of discriminants in metric vector
  spaces,
\newblock Adv. in Math., 62  (1986) 285--312.

\bibitem{MR2869185}
J.~A. Ellis-Monaghan, I.~Moffatt,
\newblock Twisted duality for embedded graphs,
\newblock Trans. Amer. Math. Soc., 364  (2012) 1529--1569.

\bibitem{MR3086663}
J.~A. Ellis-Monaghan, I.~Moffatt,
\newblock Graphs on surfaces,
\newblock Springer, New York, 2013.


\bibitem{MR1072233}
C.~P. Gabor, K.~J. Supowit, W.~L. Hsu,
\newblock Recognizing circle graphs in polynomial time,
\newblock J. Assoc. Comput. Mach., 36  (1989) 435--473.

\bibitem{GT87}
J.~L. Gross, T.~W. Tucker,
\newblock Topological graph theory,
\newblock Dover Publications, Inc., Mineola, NY, 2001.

\bibitem{KMT}
T.~Krajewski, I.~Moffatt, A.~Tanasa,
\newblock {H}opf algebras and {T}utte polynomials,
\newblock Adv. in Appl. Math., 95 (2018) 271--330.

\bibitem{bcc}
I.~Moffatt,
\newblock Delta-matroids for graph theorists,
\newblock In Surveys in combinatorics, 2019, volume XXtbcXX of London
  Math. Soc. Lecture Note Ser., page XXtbcXX. Cambridge Univ. Press, Cambridge,
  2019.

\bibitem{MS}
I.~Moffatt, B.~Smith,
\newblock Matroidal frameworks for topological tutte polynomials,
\newblock J. Combin. Theory Ser. B, 133 (2018)  1--31.

\bibitem{Oxley}
J.~Oxley,
\newblock Matroid theory,
\newblock Oxford Science Publications,  Oxford University
  Press, New York, 1992.

\bibitem{MR1850709}
J.~Oxley,
\newblock On the interplay between graphs and matroids,
\newblock In Surveys in combinatorics, 2001 ({S}ussex), volume 288 of
  London Math. Soc. Lecture Note Ser., pages 199--239. Cambridge Univ.
  Press, Cambridge, 2001.

\bibitem{MR3361421}
L.~Traldi,
\newblock The transition matroid of a 4-regular graph: an introduction,
\newblock European J. Combin., 50 (2015) 180--207.

\bibitem{MR558452}
K.~Truemper,
\newblock On {W}hitney's {$2$}-isomorphism theorem for graphs,
\newblock J. Graph Theory, 4  (1980) 43--49.

\bibitem{MR816055}
D.~K. Wagner,
\newblock On theorems of {W}hitney and {T}utte,
\newblock Discrete Math., 57  (1985) 147--154.

\bibitem{whitney}
H.~Whitney,
\newblock 2-isomorphic graphs,
\newblock Amer. J. Math., 55 (1933) 245--254, .

\end{thebibliography}
\end{document}